\documentclass[12pt,notitlepage]{article}
\usepackage{amsmath}
\usepackage{amsthm}
\usepackage{amssymb}
\usepackage{amscd}
\usepackage{amsfonts}
\usepackage{amsbsy}
\usepackage{epsfig,afterpage}
\usepackage[dvips]{psfrag}
\usepackage{psfrag}
\usepackage[all]{xy}
\usepackage{color}
\usepackage{multirow}
\usepackage{pdflscape}
\usepackage{pict2e} 
\usepackage{array}
\usepackage{comment}

\usepackage{caption}
\usepackage{xparse}
\usepackage{subcaption}
\usepackage[colorinlistoftodos]{todonotes}

\usepackage{mathtools, nccmath}

\newcommand{\floor}[1]{\left\lfloor #1 \right\rfloor}
\newcommand{\ceil}[1]{\left\lceil #1 \right\rceil}
\newcommand{\R}{\ensuremath{\mathbb{R}}}

\newcommand{\N}{\ensuremath{\mathbb{N}}}

\newcommand{\Z}{\ensuremath{\mathbb{Z}}}

\DeclareGraphicsExtensions{.eps}

\newcommand{\calH}{{\mathcal H}}
\newcommand{\calM}{{\mathcal M}}
\newcommand{\calF}{{\mathcal F}}
\usepackage{overpic}

\usepackage{todonotes}

\newtheorem {theorem} {Theorem}
\newtheorem {proposition}  {Proposition}

\newtheorem {definition} {Definition}
\newtheorem {remark} {Remark}
\newtheorem {example} {Example}
\newtheorem {corollary} {Corollary}

\AtBeginDocument{%
	
}

\usepackage{lineno,hyperref}
\modulolinenumbers[5]

\usepackage{geometry}
\usepackage{authblk}

\providecommand{\keywords}[1]
{
  \small	
  \textbf{\textit{Keywords: }} #1
}

\begin{document}

\title{Darboux-type center conditions for families of planar polynomial vector fields.}
\author[1]{Yovani Villanueva\footnote{Corresponding author.}}
\author[2]{Warwick Tucker}
\affil[1]{Universidade Federal de Goi\'as, IME, Goi\^ania, Brasil.}
\affil[2]{Monash University, School of Mathematics, Victoria, Australia.}

\maketitle	
\begin{abstract}
We study the center-focus problem for planar polynomial vector fields, which can be viewed as a local version of Hilbert's 16th problem. Based on a Lyapunov function approach, we establish novel results regarding the center-focus conditions for two families of differential systems. More precisely, we find an enclosure of the Bautin ideal generated by the Lyapunov constants of these systems. Our results hold for any degree $n\ge 2$.
\end{abstract}


\section{Introduction}\label{sec:1}
	

A fundamental problem in the theory of Dynamical Systems is Hilbert's 16th problem, posed by David Hilbert \cite{H} at the International Congress of Mathematics held in Paris in 1900. The second part of Hilbert's 16th  problem asks for the \emph{Hilbert number} $\calH(n)$ -- the maximal number of limit cycles (isolated periodic orbits) the family of planar polynomial ordinary differential equations of degree $n$ can display. Note that $\calH(n)$ should only depend on the degree $n$, not on the particular polynomial vector field itself. This question is unresolved, even for the simplest case $n = 2$. Even finding non-trivial lower bounds for $\calH(n)$ appears to be very hard \cite{TG,SS,YLE}.
	
A local version of Hilbert's 16th problem is known as the \emph{center-focus} problem. Here, vector fields whose linearizations have an equilibrium point of center type are studied, and the challenge is to find necessary and sufficient conditions (\textit{center conditions}) on the nonlinear parts of the vector field in order for the equilibrium to remain a center. When moving away from the parameters satisfying the center conditions, the equilibrium turns into a (weak) focus, and it then makes sense to ask how many small amplitude limit cycles can be generated near the equilibrium \cite{AL,MS}. This is known as the \emph{cyclicity} problem, and the maximal number of such limit cycles for degree-$n$ systems is denoted $\calM(n)$. A classical result of Bautin \cite{B,Y} states that for the class of quadratic vector fields, the maximal number of small amplitude limit cycles is three: $\calM(2) = 3$. For cubic systems without quadratic terms (e.g. homogeneous systems of degree three) \.Zol\c adek \cite{HZ,Z} proved that the maximal number of small amplitude limit cycles is five: $\calM_h(3) = 5$. Other results include lower bounds on $\calM(n)$ for $3\le n\le 10$ provided by Gouvea and Torregrosa \cite{GT}.

A common approch to these problem is to study the displacement map\footnote{The displacement map is the return map minus the identity: $d(r) = \pi(r) - r$. A zero of $d$ corrresponds to a periodic orbit of the flow. $d(r) = 0$ for all $0 \le r < r_0$ corresponds to a center.} along a radial segment originating at the fixed point. Expressed as a power series, its coefficients are polynomials in the parameters of the differential system, and are known as \emph{focal values} or \emph{Lyapunov constants}, denoted $L_j$ $(j = 1,2,\dots)$. For a given family $\calF$ of polynomial differential systems, we can form its \emph{Bautin ideal} $\mathfrak{B}(\calF) = \left< L_1, L_2, \dots \right>$ in the ring of polynomials over the set of parameters of $\calF$. The structure and size of $\mathfrak{B}(\calF)$ plays an important role in the centre-focus and cyclicity problems for the family $\calF$.


	

The goal of this paper is to establish new results regarding the Bautin ideal, valid for two important families of degree-$n$ polynomial differential equations $\calF(n)$ and $\calF_h(n)$. Our results hold for any degree $n\ge 2$, and allows us to explore some interesting classes of \emph{Darboux-type} centers. In concrete terms, we will express the Lyapunov constants in terms of a carefully selected basis. This basis allows us to give a detailed description of a set of sufficient center conditions for our two families. Our main results can be summarized as follows.
	
\textbf{Main Result:} {\it For $\calF(n)$ -- the class of (non-homogeneous) differential systems of degree $n$ -- we have:
\begin{equation*}
\mathfrak{B}(\calF(n)) = \left< L_1, L_2, \dots \right> \subset \left<v^h_3, v^h_4, \dots, v^h_{n+1} \right>.
\end{equation*}	
For $\calF_h(n)$ -- the class of homogeneous differential systems of degree $n$ -- we have:
$$
\mathfrak{B}(\calF_h(n)) = \left< L_1, L_2, \dots \right> \subset
\left\{
\begin{array}{ll}
\left< v_{n+1,0}, v_{n,1}, \dots, v_{0,n+1} \right> & \text{ if $n$ is even,}\\
\left< L_{(n-1)/2},v_{n+1,0}, v_{n,1}, \dots, v_{0,n+1} \right> & \text{ if $n$ is odd.}
\end{array}
\right. 
$$
}

The exact definitions of two families $\calF(n)$ and $\calF_h(n)$ are given by \eqref{eq:non_homogeneous_system} and \eqref{eq:homogeneous_system}, respectively. The basis elements $v^h_i$ ($i = 3,\dots, n+1$) and $v_{n+1-j,j}$ ($j = 0,\dots, n+1$) are defined by \eqref{eq_homogeneous_V_terms} and \eqref{eq:homogeneous_lyapunov_term}, respectively. An important feature of the basis elements is that they depend \emph{linearly} on the coefficients of the families under study.

The key strength of this theorem is that it gives a simple recipie for forcing all Lyapunov constants to vanish. As a consequence, we have a straight-forward way of creating a special class of centers.

\section{Preliminaries}\label{sec:2}

We will consider $\calF(n)$ -- the family real-valued planar polynomial vector fields, where the linearization at the origin describes a center, and the nonlinearities are polynomials of degree $n$. Any member of $\calF(n)$ is of the form:
\begin{equation}\label{eq:non_homogeneous_system} 
\left\{
\begin{array}{l} \displaystyle
\dot{x}=-y+\sum_{k=2}^{n}F_k(x,y),\\  \displaystyle
\dot{y}=\phantom{-}x+\sum_{k=2}^{n}G_k(x,y),
\end{array}
\right.
\end{equation}
Here the nonlinearities are represented as sums of homogeneous polynomials
\begin{equation*}
F_k(x,y)=\sum_{i+j=k}f_{i,j}x^{i}y^{j} 
\quad\textrm{ and }\quad 
G_k(x,y)=\sum_{i+j=k}g_{i,j}x^{i}y^{j}, \qquad k=2,\dots,n.
\end{equation*}
Note that the coefficients are real numbers (and not complex, which is the natural choice in other settings). We sometimes suppress the variables $x,y$ and simply write $F_k$ and $G_k$. Also, for simplicity, we will write $(f,g)$ for the set of parameters. Note that the number of parameters of \eqref{eq:non_homogeneous_system} is $n^2+3n-4$.

Determining the stability of the origin of a system \eqref{eq:non_homogeneous_system} is known as the center-focus problem. One common approach is to seek to construct a first integral around the origin. The existence of a first integral implies that any Lyapunov function (expressed in suitable coordinates) is analytic in $x^2 + y^2$. Motivated by this, we may propose $V(x,y) = \tfrac{1}{2}(x^2 + y^2) + \mathcal{O}(\|(x,y)\|^3)$ as a Lyapunov function. We can express it in its homogeneous series $\displaystyle V=V_{2}+V_{3}+V_{4}+\cdots$, with $\displaystyle V_{2}(x,y)=\tfrac{1}{2}(x^2 + y^2)$ and 
\begin{equation}\label{eq:homogeneous_lyapunov_term}
V_k(x,y) = \sum_{i+j=k} v_{i,j}x^{i}y^{j},\qquad k \geq 3. 
\end{equation} 
Along solutions of \eqref{eq:non_homogeneous_system} the formal series expansion for the derivative of $V$ (also analytic in $x^2 + y^2$) is	
\begin{eqnarray}\label{eq:lyapunov_formula}
\dot{V}(x,y) &=& \phantom{+}((V_2)'_x+(V_3)'_x+\cdots)(-y+F_{2}+F_{3}+\cdots + F_n) \nonumber\\
		& & +((V_2)'_y+(V_3)'_y+\cdots)(\phantom{-}x+G_{2}+G_{3}+\cdots +G_n) \nonumber\\
		&=& L_{1}(x^{2}+y^{2})^{2}+L_{2}(x^{2}+y^{2})^{3}+L_{3}(x^{2}+y^{2})^{4}+\cdots.
\end{eqnarray}
		
Equation \eqref{eq:lyapunov_formula} is known as the \emph{Lyapunov formula} for \eqref{eq:non_homogeneous_system}, and the coefficients $L_j$ are called \emph{Lyapunov constants}, or sometimes \textit{focal values}. Both names are somewhat deceptive; for a family of differential systems, the Lyapunov constants are in fact polynomials in the underlying coefficients: $L_j\in\R [f,g]$ (see e.g. \cite{RS}).

We can interpret the Lyapunov formula \eqref{eq:lyapunov_formula} as an infinite set of linear systems of equations; one for each degree. The unknowns to be determined\footnote{Under the assumptions $v_{(2m),(2m)}=0$ or $v_{(2m),(2m+2)} = 0$ for all $m$, the linear systems of equations have unique solutions, see \cite{DL,LK,HZ}.} are the Lyapunov constants $L_j$ for $j \ge 1$ and the $k+1$ coefficients of each homogeneous term $V_k$ for $k \geq 3$. If \emph{all} Lyapunov constants vanish, then $\dot{V}(x,y)$ is identically zero in a neighborhood of the origin, and the origin is a \emph{center}. Otherwise, if $L_1 = \dots = L_{K-1} = 0$ and $L_K\neq 0$ for some $K\in\Z^+$, the origin of \eqref{eq:non_homogeneous_system} is a \emph{weak focus} (of order $K$); its stability is given by the sign of $L_K$. 

By carefully selecting the parameters $(f,g)$ of the differential system \eqref{eq:non_homogeneous_system}, we can force all Lyapunov constants to vanish:
\begin{definition}
A \emph{center condition} is a finite set of algebraic relations of the parameters  $(f,g)$ of the differential system \eqref{eq:non_homogeneous_system} needed to ensure that $L_j = 0$ for all $j \in \Z^+$ in the Lyapunov formula \eqref{eq:lyapunov_formula}. 
\end{definition}
Normally, a center condition can be expressed as an ordered collection of algebraic equations in terms of the parameters of the vector field. The first condition ensures that $L_1 = 0$; the second condition makes $L_2 = 0$ (given that $L_1=0$), and so on. These are called \emph{partial} center conditions; each one increases the weakness of the focus at the origin. When they are all satisfied, the fixed point at the origin is a center; there are no small amplitude limit cycles. The set of all possible center conditions describes the \emph{center variety} of $\calF$. Depending on the family $\calF$ of differential systems, this can be composed of several branches; each corresponding to a particular type of center, see \cite{CC,DL,RS}.
	
\section{Lyapunov Theory for homogeneous differential systems}\label{sec:lyapunov_theory_for_homogeneous_differential_systems}

In the following few sections, we will restrict our attention to the smaller family $\calF_h(n)\subset\calF(n)$ -- consisting of all \emph{homogeneous} differential systems \eqref{eq:homogeneous_system}, where the nonlinearities are homogeneous polynomials of some fixed degree $n$: 
\begin{equation}\label{eq:homogeneous_system}
\left\{
\begin{array}{l} \displaystyle
\dot{x}=-y+F_n(x,y),\\  \displaystyle
\dot{y}=\phantom{-}x+G_n(x,y),
\end{array}
\right.
\end{equation}
As with the non-homogeneous case, we will write $(f,g)$ for the set of parameters. Note that the system \eqref{eq:homogeneous_system} has $2(n+1)$ real parameters.

Our first result makes precise statements about the minimal gaps between consecutive non-zero homogeneous terms \eqref{eq:homogeneous_lyapunov_term} of the Lyapunov function, as well as between the non-zero Lyapunov constants. 
	
\begin{proposition}\label{prop:lyapunov_gaps}
Given a homogeneous system \eqref{eq:homogeneous_system} of degree $n \ge 2$, the homogeneous terms $V_2, V_3,\dots$ of the Lyapunov function $V$ appearing in \eqref{eq:lyapunov_formula} satisfy 
$$
(V_k)_{i,j} = v_{i,j} =
\left\{
\begin{array}{llcl}
p_{i,j}(f,g) & \text{ if } k = 2 + m(n-1) \text{ for some }  m\in\N,\\
0 &\text{ otherwise.}
\end{array}
\right. 
$$
Here each $p_{i,j}$ is a (linear) polynomial in the parameters $(f,g)$. 
		
Furthermore, the Lyapunov constants $L_1,L_2,\dots$ appearing in \eqref{eq:lyapunov_formula} satisfy
\begin{equation*}
\;\;\;\;\;\;\; L_j =
\left\{
\begin{array}{llcl}
q_{j}(f,g) & \text{ if } j = m(n-1) \text{ for some } m\in\Z^+ \text{ and } n \text{ is even},\\
0 & \text{ otherwise.} & &
\end{array}
\right. 
\end{equation*} 
\begin{equation*}
L_j =
\left\{
\begin{array}{llcl}
r_{j}(f,g) & \text{ if } j = m\tfrac{n-1}{2} \text{ for some } m\in\Z^+ \text{ and } n \text{ is odd},\\
0 & \text{ otherwise.} & &
\end{array}
\right. 
\end{equation*} 
Here, each $q_j$ and $r_j$ is a (most likely non-linear) polynomial in the parameters $(f,g)$. 
\end{proposition}
	
Note that, generically, the polynomials $P_k = (p_{k,0},\dots, p_{0,k})$, $q_j$, and $r_j$ appearing in the proposition are not identically zero. The only exception to this rule is when the coefficients $(f,g)$ of $F_n$ and $G_n$ satisfy a (partial) center condition.

Before presenting the proof, we mention an important aspect of the proposition: the gaps in the sequences $\{V_k\}_{k\ge 2}$ and $\{L_j\}_{j\ge 1}$ grow with the degree $n$. Also note how the parity of the degree $n$ comes into play for the Lyapunov constants. The reason for this will become clear in the proof of the proposition.


\begin{proof}
In the homogeneous setting, the Lyapunov formula \eqref{eq:lyapunov_formula} simplifies to
\begin{eqnarray}\label{eq:homogeneous_lyapunov_formula}
\dot{V}(x,y) &=& ((V_2)'_x+(V_3)'_x+\cdots)(-y+F_n)+((V_2)'_y+(V_3)'_y+\cdots)(x+G_n) 
\nonumber\\
		     &=& L_{1}(x^{2}+y^{2})^{2}+L_{2}(x^{2}+y^{2})^{3}+L_{3}(x^{2}+y^{2})^{4}+\cdots.
\end{eqnarray}
Note that the terms of \eqref{eq:homogeneous_lyapunov_formula} can be grouped into specific blocks as follows:
\begin{eqnarray}\label{eq:homogeneous_lyapunov_equations}
\dot{V}(x,y)
&=& (x+(V_3)'_x+\cdots)(-y+F_n)+(y+(V_3)'_y+\cdots)(x+G_n) \nonumber\\
&=& \underbrace{x\sum_{k\ge3}(V_k)'_y - y\sum_{k\ge3}(V_k)'_x}_{\text{deg} = 3,4,\dots} + \underbrace{xF_n + yG_n}_{\text{deg} = n+1} + \underbrace{F_n\sum_{k\ge3}(V_k)'_x + G_n\sum_{k\ge3}(V_k)'_y}_{\text{deg} = n+2,n+3,\dots} \nonumber\\
&=& L_{1}(x^{2}+y^{2})^{2}+L_{2}(x^{2}+y^{2})^{3}+L_{3}(x^{2}+y^{2})^{4}+\cdots.
\end{eqnarray}
		
Here we see that there are three main classes of linear equations to consider, depending on the degree $k$ of the terms we are matching. 
		
\textbf{Case $3 \le k \le n$:} For low order degrees, the Lyapunov formula \eqref{eq:homogeneous_lyapunov_equations} reduces to
\begin{equation}\label{eq:small_degree}
x(V_k)'_y -y(V_k)'_x	=\left\{
\begin{array}{lcl}
L_{k/2-1}(x^{2}+y^{2})^{k/2} &\quad& \text{if } k \text{ is even},\\
0                            &\quad& \text{if } k \text{ is odd.}
\end{array}
\right.
\end{equation} 
The equation \eqref{eq:small_degree} can be cast into a system of $k+1$ linear equations; the unknowns being the coefficients of $V_k$. When $k$ is odd, the right-hand side vanishes, and it is easy to see that this system only has the trivial solution $V_k = 0$. When $k$ is even, the linear system can be split into two independent parts: one with multiples of $L_{k/2-1}$ in the right-hand side, and another with a zero right-hand side. The first system has only the trivial solution, which forces $L_{k/2-1}$ to vanish, as well as the $k/2+1$ coefficients of $V_k$ appearing. The second system of the remaining $k/2$ equations is under-determined, and admits the trivial solution. It follows that $V_k = 0$ for $k = 3,4,\dots, n$. Likewise, we have $L_j = 0$ for $j = 1,2,\dots, \floor{n/2}-1$.
		
\textbf{Case $k = n+1$:} This is the first instance where the nonlinearities $F_n$ and $G_n$ come into play. The Lyapunov formula \eqref{eq:homogeneous_lyapunov_equations} then reduces to
\begin{equation}\label{eq:mid_degree}
x(V_{n+1})'_y -y(V_{n+1})'_x + xF_n + yG_n	=\left\{
\begin{array}{lcl}
L_{(n-1)/2}(x^{2}+y^{2})^{(n+1)/2} &\quad& \text{if } n+1 \text{ is even},\\
0                                  &\quad& \text{if } n+1 \text{ is odd.}
\end{array}
\right.
\end{equation} 
Again, we cast \eqref{eq:mid_degree} into a system of $n+2$ linear equations; the unknowns being the coefficients of $V_{n+1}$. This time, however, the terms stemming from the nonlinearities $F_n$ and $G_n$ can be moved to the right-hand side. As a consequence, we no longer can expect a trivial solution. Indeed, unless the coefficients of $F_n$ and $G_n$ have been selected \emph {very} carefully\footnote{Each such non-generic selection of the parameters corresponds to a partial center condition.}, we will have a non-zero $V_{n+1}$ as the solution. Furthermore, when $n+1$ is even, we also get our first non-vanishing Lyapunov constant $L_{(n-1)/2}$.	

\textbf{Case $n+2 \le k \le 2n-1$:}
Here we are considering the degrees $k = n+i$ for $i = 2,\dots n-1$. The left-hand side of the homogeneous Lyapunov formula \eqref{eq:homogeneous_lyapunov_equations} then reduces to
\begin{equation}\label{eq:third_case_lhs}
x(V_{n+i})'_y - y(V_{n+i})'_x + F_n(V_{i+1})'_x + G_n(V_{i+1})'_y. 
\end{equation}
Since we have already established that $V_3,\dots, V_n$ all vanish, so do the two right-most terms appearing in \eqref{eq:third_case_lhs}. Therefore the full equations \eqref{eq:homogeneous_lyapunov_equations} simplify to  
\begin{equation}\label{eq:third_case}
x(V_{n+i})'_y -y(V_{n+i})'_x =		
\left\{
\begin{array}{lcl}
L_{(n+i)/2-1}(x^{2}+y^{2})^{(n+i)/2} &\quad& \text{if } n+i \text{ is even},\\
0                                    &\quad& \text{if } n+i \text{ is odd.}
\end{array}
\right.
\end{equation} 
By the exact same reasoning as for equation \eqref{eq:small_degree}, this admits only the trivial solutions: $V_k = 0$ for $k = n+2,\dots, 2n-1$, and $L_j = 0$ for $j = \ceil{n/2},\dots, \floor{(2n-1)/2}-1$.
		
Carrying on, something interesting happens. We have now reached the degree $k = 2n$ which we will treat as an extra (forth) case. 
		
\textbf{Case $k = 2n$:} Now the Lyapunov formula \eqref{eq:homogeneous_lyapunov_equations} reduces to
\begin{equation}\label{eq:fourth_case}
x(V_{2n})'_y -y(V_{2n})'_x + F_n(V_{n+1})'_x + G_n(V_{n+1})'_y	= L_{(n-1)}(x^{2}+y^{2})^{n}.
\end{equation} 
Note that, since the degree $k$ is even (independently of the parity of $n$), there are no cases in the right-hand side. We can treat \eqref{eq:fourth_case} as the case  $k=n+1$, knowing that $V_{n+1}\neq 0$ (unless a partial center condition is satisfied by $(f,g)$). Generically, we can now solve explicitly for non-zero $V_{2n}$ and $L_{n-1}$.
		
Summarizing, for the degrees $k=3,\dots 2n$, we find that the only non-zero homogeneous terms of the Lyapunov function are $V_{n+1}$ and $V_{2n}$; all other terms must vanish. We also find that if $n$ is odd, then $L_{(n-1)/2}$ is the first non-vanishing Lyapunov constant, whereas $L_{n-1}$ is non-zero independently of the parity of $n$. This establishes the proposition for the cases $m=1$ and $m=2$. 
		
It is clear that, when matching higher degree terms of the Lyapunov formula, the exact same pattern will repeat itself. Indeed, for a general $m\in\Z^+$, the key relations \eqref{eq:mid_degree} and \eqref{eq:fourth_case} with $k = 2 + m(n-1)$ can be transformed into
\begin{eqnarray}\label{eq:general_case}
x(V_{2 + m(n-1)})'_y -y(V_{2 + m(n-1)})'_x & + &(V_{2 + (m-1)(n-1)})'_xF_n + (V_{2 + (m-1)(n-1)})'_yG_n = \nonumber\\
&=& 
\left\{
\begin{array}{lcl}
L_{m\tfrac{n-1}{2}}(x^{2}+y^{2})^{m\tfrac{n-1}{2}+1} &\quad& \text{if } m(n-1) \text{ is even},\\
0                                  &\quad& \text{if } m(n-1) \text{ is odd.}
\end{array}
\right.
\end{eqnarray}
Note that we can only use (\ref{eq:general_case}) to solve for a Lyapunov constant when the product $m(n-1)$ is even. There are two scenarios here: if $n$ is odd, then $m(n-1)$ is always even, and we can solve for the Lyapunov constant with index $m\tfrac{n-1}{2}$ for each $m\in\Z^+$. If, on the other hand, $n$ is even, then $m(n-1)$ is even only when $m$ is. In this situation, we can only use (\ref{eq:general_case}) to solve for the Lyapunov constant with index $m\tfrac{n-1}{2}$ for $m \in 2\Z^+$.
		
All remaining terms of degree $k = 2 + m(n-1) + i$, $i = 1,\dots n-2$ can be considered by analogues to (\ref{eq:third_case_lhs}) and (\ref{eq:third_case}). This leads to only trivial (zero) solutions for $V_k$ and the corresponding Lyapunov constants $L_j$. In summary, the statements of the proposition give a complete account of the only non-zero $V_k$ and $L_j$ in terms of the degree $n$ and the parameters $(f,g)$ of the nonlinearities $F_n$ and $G_n$ of the homogeneous system \eqref{eq:homogeneous_system}.
\end{proof} 
	


As a direct consequence of the proposition, we see that the first non-zero Lyapunov constants follow a distinct pattern, which depends on the parity of $n$. Another consequence of the proof of the proposition is that we have uncovered an explicit algebraic structure of the non-zero Lyapunov constants. 
We capture this finding in the following corollary:

\begin{corollary}\label{cor:vn_explicit_in_Lj}
Given a homogeneous system \eqref{eq:homogeneous_system} of degree $n \ge 2$, generically all $n+2$ coefficients of $V_{n+1}$ appear explicitly in the expression of each non-zero Lyapunov constant $L_j$, for $j \geq n-1$. Furthermore, each coefficient of $V_{n+1}$ depends linearly on the underlying parameters $(f,g)$.
\end{corollary}

Note that for odd degrees $n$, the coefficients of $V_{n+1}$ do not appear in the algebraic expression for the very first non-zero Lyapunov constant $L_{(n-1)/2}$. That is why we have added the restriction $j \geq n-1$ in the corollary. We point out that this is the only exception to the rule, and it follows from observing the structure of \eqref{eq:mid_degree}. A more detailed explanation is given in the Appendix.

\section{Main theorems for homogeneous differential systems}\label{sec:main_theorem_for_generic_homogeneous_differential_systems}
	
Now that we understand the sparsity patterns of the non-zero Lyapunov constants $L_j$ and the homogeneous terms $V_k$ of the Lyapunov function, we can turn our attention to the Bautin ideal of the family.

We will need a few more definitions to make our results precise.	
\begin{definition}
Given a finitely parameterized family of differential equations $\calF$, its \emph{Bautin ideal} $\mathfrak{B}(\calF)$ is the ideal generated by the Lyapunov constants, $\left< L_1, L_2, \dots\right>$, in the ring of polynomials over the set of parameters of $\calF$.
\end{definition}


It is well-known that, for both families $\calF(n)$ and $\calF_h(n)$, all Lyapunov constants are polynomials\footnote{When adding a non-zero trace term to a member of such a family, the Lyapunov constants become analytic functions in the parameters $(f,g)$; not merely polynomials.} in the parameters of the family. Therefore, it follows (by Hilbert's Basis Theorem) that, for either family $\calF$, the Bautin ideal is finitely generated: there exists a positive integer $K = K(\calF)$ such that $\mathfrak{B}(\calF) = \left< L_1, L_2, \dots \right> = \left< L_1, L_2, \dots, L_K\right>$. As a consequence, this is also the maximal bound on the order of a weak focus. Beyond order $K$, any weak focus becomes a center. In light of this, we call $K(\calF)$ the \emph{center number} of the family $\calF$.

The center number of a family is very useful to know. In order to verify the existence of a center, we only need to compute the $K$ first (non-trivial) Lyapunov constants. If they all vanish, we have a center; if not, we have a (weak) focus. Of course, even knowing an upper bound on the center number can be used to explicitly verify a center. Unfortunately, for general families there are no effective methods to compute their center numbers, or any useful upper bounds. We can, however, \emph{enclose} the Bautin ideal $\mathfrak{B}$ by a finitely generated ideal $\mathfrak{V}$, which we understand in detail.

\begin{theorem}\label{thm:homogeneous_bautin}
For even degrees $n$, the ideal generated by the coefficients of $V_{n+1}$ contains the Bautin ideal of $\calF_h(n)$:
$$
\mathfrak{B}(\calF_h(n)) = \left< L_1, L_2, \dots \right> \subseteq \left< v_{n+1,0}, v_{n,1}, \dots, v_{0,n+1} \right>,
$$
and $V_{n+1}=0$ gives a collection of center conditions. For odd degrees $n$, we have 
$$
\mathfrak{B}(\calF_h(n)) = \left< L_1, L_2, \dots \right> \subseteq \left< L_{(n-1)/2},v_{n+1,0}, v_{n,1}, \dots, v_{0,n+1} \right>,
$$
and $L_{(n-1)/2} = 0$ together with $V_{n+1}=0$ provide a collection of center conditions.
\end{theorem}
	
Note that this theorem provides a finite number of explicit (linear) constraints on the underlying parameters $(f,g)$ \emph{sufficient} to create a center for a homogeneous system \eqref{eq:homogeneous_system}. Also note, however, that the Lyapunov constants can vanish independently of the coefficients of $V_{n+1}$. There are center conditions that are not captured by $V_{n+1} = 0$.

\begin{proof}[Proof of Theorem \ref{thm:homogeneous_bautin}:]

First, we point out that the reason we have to add $L_{(n-1)/2}$ to the ideal in the odd case is directly related to the discussion following Corollary~\ref{cor:vn_explicit_in_Lj}. The linear system that arises from \eqref{eq:mid_degree} allows for a non-zero $L_{(n-1)/2}$ whilst all coefficients of $V_{n+1}$ vanish. 

Moving on, by Proposition \ref{prop:lyapunov_gaps}, the first non-vanishing homogeneous term of the Lyapunov function after $V_2$ is $V_{n+1}$, having $n+2$ coefficients. Rather than using \eqref{eq:mid_degree} to solve directly for $V_{n+1}$, we keep its coefficients as variables in the expressions for higher order terms $V_{2+m(n-1)}$ (for all $m \geq 2$) and Lyapunov constants $L_{m(n-1)/2}$ (when $m(n-1)$ is even). Using equation \eqref{eq:general_case} we will see that all solutions can be expressed as linear functions of the coefficients of $V_{n+1}$. To prove this fact, we will use induction over $m$ applied to equation \eqref{eq:general_case}. 
From here on we fix the degree $n$, and use the notation $v_k = (v_{k,0},\dots,v_{0,k})^T$ to denote the vector of coefficients of the homogeneous term $V_k$ defined in \eqref{eq:homogeneous_lyapunov_term}.
		
\textbf{Base case $m = 2$:} Note that $m(n-1) = 2(n-1)$ is even for all degrees $n$, so by rearranging equation \eqref{eq:general_case} we have
\begin{equation*}
L_{n-1}(x^{2}+y^{2})^{n}+y(V_{2n})'_x-x(V_{2n})'_y
=
(V_{n+1})'_xF_n+(V_{n+1})'_yG_n \stackrel{\text{def}}{=} P_2(f,g,x,y) \cdot v_{n+1}
\end{equation*}      
for some suitable vector $P_2$ of polynomials in the parameters and variables. Matching powers in the variables $(x,y)$, and solving, it follows immediately that $L_{n-1}$ and all coefficients of  $V_{2n}$ are linear combinations of those of $V_{n+1}$. In other words, from $P_2$ we can derive a new vector $P_2^L$ (of length $n+2$) of polynomials, together with a matrix $P_2^V$ (of size $(2n+1)\times (n+2)$) of polynomials such that we have 
$$
L_{n-1} = P_2^L(f,g) \cdot v_{n+1} 
\quad\textrm{and}\quad
v_{2n} = P_2^V(f,g) \cdot v_{n+1}.
$$
We emphasise that $P_2^L$ and $P_2^V$ only depend on the parameters $(f,g)$, not on the variables $(x,y)$. By a slight abuse of notation, we can now write $V_{2n} = P_2^V(f,g,x,y)\cdot v_{n+1}$, where $P_2^V(f,g,x,y)$ is a vector of polynomials in both parameters $(f,g)$ and variables $(x,y)$:
$$
P_2^V(f,g,x,y) = (x^{2n},x^{2n-1}y,\dots, xy^{2n-1}, y^{2n})\cdot P_2^V(f,g). 
$$

\textbf{General case $m>2$:} Suppose that $V_{2+(m-1)(n-1)}$ is linear function of the coefficients of $V_{n+1}$:
$$
V_{2+(m-1)(n-1)}=P_{m-1}^V(f,g,x,y) \cdot v_{n+1}.
$$ 
We now have two cases to consider, depending on the parity of the product $m(n-1)$.

If $m(n-1)$ is odd, we can rearrange and simplify equation \eqref{eq:general_case} as    
\begin{align*}
y(V_{2 + m(n-1)})'_x &- x(V_{2 + m(n-1)})'_y = (V_{2 + (m-1)(n-1)})'_xF_n + (V_{2 + (m-1)(n-1)})'_yG_n \\
&= [(P_{m-1}^V(f,g,x,y))'_xF_n] \cdot v_{n+1} + [(P_{m-1}^V(f,g,x,y))'_yG_n] \cdot v_{n+1} \\
&\stackrel{\text{def}}{=} P_{m}(f,g,x,y) \cdot v_{n+1},
\end{align*}
and it is clear that $V_{2 + m(n-1)}$ depends linearly on $v_{n+1}$.

When $m(n-1)$ is even, we can repeat the above, but we also solve for the next Lyapunov constant:
\begin{eqnarray}\label{eq:long_eq}
L_{m\tfrac{n-1}{2}}(x^{2} &+& y^{2})^{m\tfrac{n-1}{2}+1} + y(V_{2 + m(n-1)})'_x - x(V_{2 + m(n-1)})'_y\nonumber\\
&=& (V_{2 + (m-1)(n-1)})'_xF_n + (V_{2 + (m-1)(n-1)})'_yG_n \nonumber\\
&=& [(P_{m-1}^V(f,g,x,y))'_xF_n] \cdot v_{n+1} + [(P_{m-1}^V(f,g,x,y))'_yG_n] \cdot v_{n+1}\nonumber\\
&\stackrel{\text{def}}{=}& P_{m}(f,g,x,y) \cdot v_{n+1}.
\end{eqnarray}
It follows that both $V_{2 + m(n-1)}$ and $L_{m\tfrac{n-1}{2}}$ depend linearly on $v_{n+1}$. This concludes the induction step, and the statement on the Bautin ideal follows. As a consequence, if the coefficients of $V_{n+1}$ vanish, so will all other homogeneous terms of the Lyapunov function, together with the Lyapunov constants (with the exception of $L_{(n-1)/2}$ when $n$ is odd). 
\end{proof}



We now explain what type of centers that are naturally represented by our choice of basis. Centers come in many guises: Hamiltonian, reversible, and (Darboux) integrable. The ones we can find are of Darboux-type. Note that $v_{n+1}=0$ if and only if $V_{n+1}=0$ for all $x,y \in \R$.
\begin{proposition}\label{prop:darboux_centers}
For all degrees $n \geq 2$, the centers defined by $v_{n+1}=0$ for even $n$, and by $L_{(n-1)/2}=0$ and $v_{n+1}=0$ for odd $n$, correspond to Darboux centers. 
\end{proposition}

\begin{proof}
In order to establish the presence of a Darboux center, we need to find an invariant\footnote{Technically speaking, $\gamma$ itself is usually not constant along trajectories, but its logarithmic derivative is simple: $\tfrac{d}{dt}\log \gamma(x(t),y(t)) = K(x(t),y(t))$.} algebraic curve $\gamma$ together with a cofactor $K$, such that 
\begin{equation}\label{eq:inv_alg_curve}
\dot{x}\frac{\partial \gamma}{\partial x} + \dot{y}\frac{\partial \gamma}{\partial y}=K \gamma.
\end{equation}

Assuming that $V_{n+1} = 0$ (and that $L_{(n-1)/2}=0$ for $n$ odd), it follows from \eqref{eq:mid_degree} that $xF_n + yG_n = 0$.  Therefore, given a homogeneous system \eqref{eq:homogeneous_system}, we have
\begin{equation*}
x\dot{x} + y\dot{y} = -xy + xF_n + xy + yG_n = 0,
\end{equation*}
so all solutions stay on concentric circles centred at the origin. In this case, we can select $\gamma_1 = x^2 + y^2$, and $K_1 = 0$ in \eqref{eq:inv_alg_curve}.

Another consequence of $xF_n + yG_n = 0$, is that $F_n$ contains no $x^n$-term, and $G_n$ contains no $y^n$-term. This means that we can factor $F_n = y\tilde {F}_n$ and $G_n = x\tilde{G}_n$, and it follows that
\begin{equation*}
0 = xF_n + yG_n = xy\tilde{F}_n + xy\tilde{G}_n = xy(\tilde{F}_n + \tilde{G}_n).
\end{equation*}
Assuming that $xy\neq 0$, we therefore have $\tilde{F}_n + \tilde{G}_n = 0$, or in other words $G_n = x\tilde{G}_n = -x\tilde{F}_n$. We can therefore write our homogeneous system \eqref{eq:homogeneous_system} in the form
\begin{equation}\label{eq:darbouxgen}
\left\{
\begin{array}{l} \displaystyle
\dot{x}=-y + y\tilde{F}_n = -y(1 - \tilde{F}_n),\\ 
\dot{y}=\phantom{-}x -x\tilde{F}_n  = \phantom{-}x(1 - \tilde{F}_n).
\end{array}\right.
\end{equation}
We now see a new candidate for an invariant algebraic curve: $\gamma_2 = (1 - \tilde{F}_n)$. Solving for its cofactor $K_2$ using \eqref{eq:inv_alg_curve} and \eqref{eq:darbouxgen}, we arrive at
$$
\dot{x}\frac{\partial \gamma_2}{\partial x} + \dot{y}\frac{\partial \gamma_2}{\partial y}
=
-y(1 - \tilde{F}_n)(-\tilde{F}_n)'_x + x(1 - \tilde{F}_n)(-\tilde{F}_n)'_y
=
\left(y(\tilde{F}_n)'_x - x(\tilde{F}_n)'_y\right)(1 - \tilde{F}_n),
$$
so $K_2 = y(\tilde{F}_n)'_x - x(\tilde{F}_n)'_y$. Note that $K_2$ equals the divergence of the system \eqref{eq:darbouxgen}:
$$
\frac{\partial \dot{x}}{\partial x}+\frac{\partial \dot{y}}{\partial y}
=
y(\tilde{F}_n)'_x - x(\tilde{F}_n)'_y
=
K_2.
$$
By Darboux's Theorem of Integrability \cite[Chapter 8]{DL}, there exists complex numbers $\lambda_1, \lambda_2$ such that $\lambda_1 K_1 + \lambda_2 K_2 = -K_2$ if and only if 
$\gamma_1^{\lambda_1}\gamma_2^{\lambda_2}$ is an integrating factor of \eqref{eq:darbouxgen}.
Since $K_1 = 0$, we can select any $\lambda_1$, whilst we must select $\lambda_2 = -1$. Therefore, any function
$$
R = \frac{(x^2 + y^2)^{\lambda_1}}{1 - \tilde{F}_n}
$$
is an integrating factor for \eqref{eq:darbouxgen}, with the corresponding first integral
$$
H = -\frac{(x^2 + y^2)^{\lambda_1 + 1}}{2(\lambda_1 + 1)}\quad(\lambda_1\neq -1),\qquad  
H=\tfrac{1}{2}\ln(x^2+y^2)\quad (\lambda_1 = -1).
$$
\end{proof}

\section{Non-homogeneous differential systems}\label{sec:non_homogeneous_differential_systems}

In this section, we will consider the larger family \eqref{eq:non_homogeneous_system} of \emph{non-homogeneous} polynomial vector fields of degree $n$ with parameters in $\mathbb{R}$. Compared to the homogeneous case, the number of parameters is much larger: $n^2+3n-4$ instead of $2(n+1)$. Another difference is that we no longer observe gaps in the sequence of Lyapunov constants; there is no equivalent to Proposition~\ref{prop:lyapunov_gaps}. This is a direct consequence of the Lyapunov equation \eqref{eq:lyapunov_formula} now containing the full nonlinearities of the differential system \eqref{eq:non_homogeneous_system}.

In this more general setting, additional terms appear in the equations governing both the coefficients of the Lyapunov function and the Lyapunov constants. Theorem \ref{thm:homogeneous_bautin} can be extended by observing that the homogeneous terms $v_k$ $(k = 2,\dots,n)$ of the Lyapunov function can be decomposed as follows:
\begin{eqnarray}\label{eq_homogeneous_V_terms}
v_{3} &=& v^h_{3}(f_2,g_2), \nonumber\\	
v_{4} &=& v^h_{4}(f_3,g_3) + U_{4}(v_3), \nonumber\\
v_{5} &=& v^h_{5}(f_4,g_4) + U_{5}(v_3,v_4), \\
&\vdots& \nonumber\\
v_{n+1} &=& v^h_{n+1}(f_n,g_n) + U_{n+1}(v_3, v_4, \dots, v_n)\nonumber.
\end{eqnarray}
The notation here is a bit involved, so let us pause and explain it in detail. First, we are using the shorthand notation $v_k$ to denote the coefficients of the homogeneous term $V_k$, so $v_k = (v_{k,0}, v_{k-1,1},\dots, v_{0,k})^T$. Second, we are using the shorthand notation $(f_k, g_k)$ to denote the coefficients of the homogeneous nonlinearities $(F_k, G_k)$, so $f_k = (f_{k,0}, f_{k-1,1},\dots, f_{0,k})$, and similarly for $g_k$.

In \eqref{eq_homogeneous_V_terms}, the terms $v^h_k$ correspond to the terms obtained by only studying the homogeneous differential system of degree $k$. An important feature, which was pointed out in Corollary~\ref{cor:vn_explicit_in_Lj}, is that each $v^h_{k+1}$ is \emph{linear} in the parameters $(f_{k}, g_{k})$. The terms $U_k$ capture the effects of the non-homogenous terms in the differential system, and we will see in Theorem \ref{thm:non_homogeneous_bautin} that
$$
U_{k+1}(v_3, v_4, \dots, v_k) = \sum_{j \leq k}P_j(f_2,g_2,\dots,f_{k-1},g_{k-1}) \cdot v^h_{j}.
$$
The whole point of the decomposition \eqref{eq_homogeneous_V_terms} is that each $V_k$ can be expressed in terms of $v^h_3,\dots, v^h_k$, each of which we understand well from the previous sections.
	
\begin{theorem}\label{thm:non_homogeneous_bautin}
For any degree $n\ge 2$, the ideal generated by the homogeneous coefficients of $v_{3},\dots,v_{n+1}$, defined by \eqref{eq_homogeneous_V_terms}, contains the Bautin ideal of $\calF(n)$:
\begin{equation*}
		\mathfrak{B}(\calF(n)) = \left< L_1, L_2, \dots \right> \subseteq \left<v^h_3, v^h_4, \dots, v^h_{n+1} \right>,
\end{equation*}
and $v^h_3 = \dots = v^h_{n+1} = 0$ gives a collection of center conditions.
\end{theorem}
	
Analogously to Theorem~\ref{thm:homogeneous_bautin}, this provides us with \emph{sufficient} conditions for having a center, not necessary ones.
	
\begin{proof}
Consider equation \eqref{eq:non_homogeneous_system} with only two homogeneous nonlinearities of fixed degrees $k$ and $n$, with $k < n$. The Lyapunov equation then reduces to 
		
\begin{eqnarray}\label{eq:lyapunov_formula2}
\dot{V}(x,y) &=& \phantom{+}((V_2)'_x+(V_3)'_x+\cdots)(-y+F_{k}+ F_n) \nonumber\\
  			& & +((V_2)'_y+(V_3)'_y+\cdots)(\phantom{-}x+G_{k}+G_n) \nonumber\\
			&=& \phantom{+}(x+(V_3)'_x+\cdots)(-y+F_{k})+(x+(V_3)'_x+\cdots)F_n \nonumber\\
			& & +(y+(V_3)'_y+\cdots)(\phantom{-}x+G_{k})+(y+(V_3)'_y+\cdots)G_n \nonumber\\
			&=& L_{1}(x^{2}+y^{2})^{2}+L_{2}(x^{2}+y^{2})^{3}+L_{3}(x^{2}+y^{2})^{4}+\cdots.
\end{eqnarray}	
			
By following the proof of Proposition \ref{prop:lyapunov_gaps}, we see that $V_{k+1}$ does not depend on the higher order nonlinearities $(F_n,G_n)$:
\begin{equation}\label{eq:mid_degreek}
x(V_{k+1})'_y -y(V_{k+1})'_x + xF_k + yG_k	=\left\{
\begin{array}{lcl}
L_{(k-1)/2}(x^{2}+y^{2})^{(k+1)/2} &\quad& \text{if } k+1 \text{ is even},\\
				0                  &\quad& \text{if } k+1 \text{ is odd.}
\end{array}
\right.
\end{equation} 
Indeed, since $xF_k + yG_k$ is a linear expression in terms of the parameters $(f_k,g_k)$, so is the term $V_{k+1}$. Furthermore, we have $v_{k+1}=v_{k+1}^h$, i.e., $U_{k+1} = 0$.
			
In contrast, the term $V_{n+1}$ will have a more complicated dependence on the parameters of the system. Here, we will have to consider contributions from the terms
$[(V_{n+2-k})'_xF_k+(V_{n+2-k})'_yG_k]$ that are of degree $n+1$. By Theorem \ref{thm:homogeneous_bautin}, we have
$$
[(V_{n+2-k})'_xF_k+(V_{n+2-k})'_yG_k]=P_{n+2-k}(f,g,x,y)\cdot v^h_{k+1}.
$$ 		

We will now show that the coefficients of the Lyapunov function and the Lyapunov constants appearing in \eqref{eq:lyapunov_formula2} are linear combinations of $v^h_{k+1}$ and $v^h_{n+1}$.

		Taking equation \eqref{eq:long_eq}, with the Lyapunov formula \eqref{eq:lyapunov_formula2}, if $m(n-1)$ is odd and there exists $\bar{m}$ such that $m(n-1)=\bar{m}(k-1)$, we can rearrange and simplify equation \eqref{eq:general_case} as    
		\begin{eqnarray}\label{eq:long_eq1}
			y(V_{2 + m(n-1)})'_x - x(V_{2 + m(n-1)})'_y &=& (V_{2 + (m-1)(n-1)})'_xF_n + (V_{2 + (m-1)(n-1)})'_yG_n \nonumber\\
			& & +(V_{2 + (\bar{m}-1)(k-1)})'_xF_k + (V_{2 + (\bar{m}-1)(k-1)})'_yG_k \nonumber\\
			&=& P_{m}(f,g,x,y) \cdot v^h_{n+1}+P_{\bar{m}}(f,g,x,y) \cdot v^h_{k+1},
		\end{eqnarray}
		and it is clear that $V_{2 + m(n-1)}$ is a linear combination of $v^h_{n+1}$ and $v^h_{k+1}$. If no such $\bar{m}$ exists, then $V_{2 + m(n-1)}$ is a linear combination of $v^h_{n+1}$ only.

		When $m(n-1)$ is even, we can repeat the above, but we also solve for the next Lyapunov constant:
		\begin{eqnarray}\label{eq:long_eq2}
			L_{m\tfrac{n-1}{2}}(x^{2} &+& y^{2})^{m\tfrac{n-1}{2}+1} + y(V_{2 + m(n-1)})'_x - x(V_{2 + m(n-1)})'_y\nonumber\\
			&=& (V_{2 + (m-1)(n-1)})'_xF_n + (V_{2 + (m-1)(n-1)})'_yG_n \nonumber\\
			& & (V_{2 + (\bar{m}-1)(k-1)})'_xF_k + (V_{2 + (\bar{m}-1)(k-1)})'_yG_k\nonumber\\
			&=& P_{m}(f,g,x,y) \cdot v^h_{n+1} + P_{\bar{m}}(f,g,x,y) \cdot v^h_{k+1}.
		\end{eqnarray}
		
		It follows that both $V_{2 + m(n-1)}$ and $L_{m\tfrac{n-1}{2}}$ are linear combinations of $v^h_{n+1}$ and $v^h_{k+1}$. 

Adding an additional (third) nonlinearity of degree $j$, where $j\in\{2,\dots,n\}$ and $j \neq k$, $j \neq n$, the procedure is analogous to the one described above. Iterating this process for each degree up to $n$, we conclude that every coefficient of the Lyapunov function and Lyapunov constant can be expressed as a linear combination of $v^h_{3},\; v^h_4, \cdots, \; v^h_{n+1}$. This concludes the theorem.
\end{proof}

We conclude this section by describing the type of centers that we capture by setting $v^h_3 = \dots = v^h_{n+1} = 0$. It should come as no surprise that they too are of Darboux type.

\begin{proposition}\label{prop:teodarbgen}
	Centers defined by $v^h_3=v^h_4=\cdots=v^h_{n+1}=0$  correspond to Darboux's centers.
\end{proposition}

\begin{proof}
	
Following the proof of Theorem \ref{thm:non_homogeneous_bautin}, we begin with a differential system having homogeneous nonlinearities of only two fixed degrees $k$ and $n$, $k<n$. By rearranging \eqref{eq:mid_degreek} we have
\begin{equation}\label{eq:mid_degreek2}
xF_k + yG_k	=\left\{
\begin{array}{lcl}
L_{(k-1)/2}(x^{2}+y^{2})^{(k+1)/2}+y(V_{k+1})'_x -x(V_{k+1})'_y &\quad& \text{if } k+1 \text{ is even},\\
y(V_{k+1})'_x -x(V_{k+1})'_y                  &\quad& \text{if } k+1 \text{ is odd.}
\end{array}\right.
\end{equation} 
Noting that $V_{k+1}^h=V_{k+1}$, the assumtion $v_{k+1}^h=0$ translates into $V_{k+1} = 0$
(which also yields $L_{(k-1)/2}=0$ for $k$ odd), and it follows from \eqref{eq:mid_degreek2} that $xF_k + yG_k=0$.	

As in Proposition \ref{prop:darboux_centers}, we find that $G_k = x\tilde{G}_k = -x\tilde{F}_k$, and the associated homogeneous system can be expressed as
\begin{equation}\label{eq:darbouxgen2}
\left\{
\begin{array}{l} \displaystyle
\dot{x}= -y(1 - \tilde{F}_k),\\ 
\dot{y}= \phantom{-}x(1 - \tilde{F}_k).
\end{array}\right.
\end{equation}

Turning to the higher (degree-$n$) nonlinearity, using \eqref{eq:long_eq1} and \eqref{eq:long_eq2} with $m=1$, together with the already established $V_{k+1}^h=V_{k+1}=0$, we have
\begin{equation}\label{eq:mid_degreek3}
	xF_n + yG_n	=P_{1}(f,g,x,y) \cdot v^h_{n+1}.
\end{equation} 
Since we are assuming that $V_{n+1}^h=0$, it follows that $xF_n + yG_n=0$. By the same reasoning as in the proof of Proposition~\ref{prop:darboux_centers} we have $G_n = x\tilde{G}_n = -x\tilde{F}_n$, and our differential system can be simplified to
\begin{equation}\label{eq:darbouxgen3}
	\left\{
	\begin{array}{l} \displaystyle
		\dot{x}= -y(1 - \tilde{F}_k- \tilde{F}_n),\\ 
		\dot{y}= \phantom{-}x(1 - \tilde{F}_k- \tilde{F}_n).
	\end{array}\right.
\end{equation}

Repeating the previous process for nonlinearities of all remaining degrees, we can express the associated system for $V^h_3=V^h_4=\cdots=V^h_{n+1}=0$ as follows:   
\begin{equation}\label{eq:darbouxsimk}
	\left\{
	\begin{array}{l} \displaystyle
		\dot{x}=-y(1 -\tilde{F}_2 -\tilde{F}_3 \dots -\tilde{F}_n),\\ 
		\displaystyle		\dot{y}=\phantom{+}x(1 -\tilde{F}_2 -\tilde{F}_3 \dots -\tilde{F}_n).
	\end{array}\right.
\end{equation}	
Similarly to Proposition \ref{prop:darboux_centers}, we can find invariant algebraic curves $\gamma_1 = x^2 + y^2$, with vanishing cofactor $K_1$, and $\gamma_2=1- \tilde{F}_2- \tilde{F}_3- \dots - \tilde{F}_n$ with cofactor $K_2 = y(\tilde{F}_2+\tilde{F}_3+ \dots +\tilde{F}_n)'_x - x(\tilde{F}_2+\tilde{F}_3+ \dots +\tilde{F}_n)'_y$.
As in Proposition~\ref{prop:darboux_centers}, for $\lambda_1\in\R$, we have the following first integrals
$$
H=-\frac{(x^2 + y^2)^{\lambda_1 + 1}}{2(\lambda_1 + 1)} \quad (\lambda_1 \neq -1), \qquad H=\tfrac{1}{2}\ln(x^2 + y^2) \quad (\lambda_1 = -1),
$$
and the origin for \eqref{eq:darbouxsimk} is a Darboux center.

\end{proof}


\section{Conclusions}\label{sec:conclusions}
	
We have established concrete enclosures of the Bautin ideal for two families of planar polynomial differential systems of degree $n$, having a center-focus type fixed-point at the origin. The key idea is to move from a system's original parameters to the coefficients of the homogeneous terms of the associated Lyapunov function. The obtained description provides us with (sufficient) center conditions for the families.

A clear advantage of our approach is that the results are valid for all degrees. Also, the simple representation provides a concrete pathway to verify the existence of a center of Darboux-type. A shortcoming is that we do not get an exact description of the Bautin ideal, so there are center conditions that we cannot account for. This restriction prevents us from exploring carefully crafted systems with (potentially) large cyclicity. Indeed, describing the cyclicity of general centers is still an open problem for all degrees larger than two, and it is clear that such cases require a much more detailed analysis.

\subsection{Acknowledgments}
	
Both authors were supported by the grant ARC DP220100492. In addition, the first author was partially supported by the grants PDSE-CAPES 88881.624523/2021-01 and DS-CAPES 88882.386238/2019-01. We are grateful to Prof. Joan Torregrosa from Universitat Autònoma de Barcelona, for all the fruitful conversations in refining this theory.

{\footnotesize  ``Elemental does not mean trivial.'' Prof. Jes\'us Hernando P\'erez, Pelusa.}
	
\bibliographystyle{abbrv}

\bibliography{refs}
	

\textbf{Appendix}

Here we explain in more detail why -- for odd degrees $n$ -- the coefficients of $V_{n+1}$ do not appear in the first non-trivial Lyapunov constant $L_{(n-1)/2}$. We will illustrate the underlying mechanism for the case $n=3$; then \eqref{eq:mid_degree} becomes
\begin{equation*}
x(V_4)'_y - y(V_4)'_x + xF_3 + yG_3	= L_1(x^{2}+y^{2})^{2}, 
\end{equation*} 
which can be rearranged into 
\begin{equation}\label{eq:mid_degree_n_is_three}
L_1(x^{2}+y^{2})^{2} - x(V_4)'_y + y(V_4)'_x = xF_3 + yG_3. 
\end{equation} 
Matching the terms on both sides of \eqref{eq:mid_degree_n_is_three} produces a linear system of five equations in six unknowns: $v_{4,0},\dots,v_{0,4}$, and $L_1$: 
\begin{displaymath}
\begin{pmatrix}
0 &-1 & 0 & 0 & 0 & 1  \\
4 & 0 &-2 & 0 & 0 & 0  \\
0 & 3 & 0 &-3 & 0 & 2  \\
0 & 0 & 2 & 0 &-4 & 0  \\
0 & 0 & 0 & 1 & 0 & 1  
\end{pmatrix} 
\begin{pmatrix}
v_{4,0} \\
v_{3,1} \\
v_{2,2} \\
v_{1,3} \\
v_{0,4} \\
L_1 
\end{pmatrix}
=
\begin{pmatrix}
f_{3,0} \\
f_{2,1} + g_{3,0} \\
f_{1,2} + g_{2,1} \\
f_{0,3} + g_{1,2} \\
g_{0,3} 
\end{pmatrix}.
\end{displaymath}
We can split this system into two smaller ones. The equations for $L_1$, $v_{3,1}$ and $v_{1,3}$ are
\begin{displaymath}
\begin{pmatrix}
-1 & 0 & 1 \\
 3 &-3 & 2 \\
 0 & 1 & 1 \\
\end{pmatrix} 
\begin{pmatrix}
v_{3,1} \\
v_{1,3} \\
L_1
\end{pmatrix}
=
\begin{pmatrix}
f_{3,0} \\
f_{1,2} + g_{2,1} \\
g_{0,3} 
\end{pmatrix}.
\end{displaymath}
Seeing that the appearing matrix is non-singular, this always produces a unique solution. Here it also becomes clear that $L_1$ is completely decoupled from the remaining three coefficients of $V_4$. The (underdetermined) equations for $v_{4,0}, v_{2,2}$ and $v_{0,4}$ are:
\begin{displaymath}
\begin{pmatrix}
 4 &-2 & 0 \\
 0 & 2 &-4 \\
\end{pmatrix} 
\begin{pmatrix}
v_{4,0} \\
v_{2,2} \\
v_{0,4} 
\end{pmatrix}
=
\begin{pmatrix}
f_{2,1} + g_{3,0} \\
f_{0,3} + g_{1,2} 
\end{pmatrix},
\end{displaymath}
so we have a one-parameter family of solutions, say $\big(v_{4,0}(t), v_{0,4}(t)\big)$, where the real-valued parameter $t$ takes the role of $v_{2,2}$. By fixing the parameter $t$ (and hence $v_{2,2}$), we have unique solutions for $v_{4,0}$ and $v_{0,4}$, and therefore also for $V_4$. This was used in the proof of Proposition~\ref{prop:lyapunov_gaps}, where we selected $t = 0$ throughout.  

It is clear that the above situation will occur for all odd degrees $n$. For such $n$, \eqref{eq:mid_degree} will produce a system $n+2$ linear equations for the $n+3$ unknowns (the $n+2$ coefficients of $V_{n+1}$ together with $L_{(n-1)/2}$). As we have just seen, this system can be decoupled into two independent systems, one with a unique solution, and one which is underdetermined, producing a one-parameter family of solutions. Depending on the parity of $(n+1)/2$, a natural\footnote{When $(n+1)/2 = 2m+1$, the equation $-v_{2m,2m+2} + v_{2m+2,2m} = (f_{2m, 2m+1} + g_{2m+1,2m})/(2m+2)$ provides the parametrization needed to make the full system uniquely solvable.} choice for the real-valued parameter is to take 
\begin{equation}\label{eq:free_parameter}
t  =		
\left\{
\begin{array}{lcl}
v_{2m,2m}   &\quad& \text{if } (n+1)/2 = 2m,\\
v_{2m,2m+2} &\quad& \text{if } (n+1)/2 = 2m + 1.
\end{array}
\right.
\end{equation}

\end{document}